\input ./style/arxiv-vmsta.cfg
\documentclass[numbers,compress,v1.0.1]{vmsta}

\volume{3}
\issue{1}
\pubyear{2016}
\firstpage{95}
\lastpage{103}
\doi{10.15559/16-VMSTA49}% Updated by VTEXPTS2LaTeX.exe, 16.03.2016
\startlocaldefs

\newcommand{\rrvert}{\vert}
\newcommand{\llvert}{\vert}
\allowdisplaybreaks

\newtheorem{thm}{Theorem}

\newtheorem{lem}{Lemma}

\newcommand{\eps}{\varepsilon}

\renewcommand{\Re}{\mathbb{R}}
\newcommand{\be}{\begin{equation}
} \newcommand{\ee} {
\end{equation}}
\newcommand{\ba}{\begin{aligned}}
\newcommand{\ea}{\end{aligned}}

\renewcommand{\P}{\mathbf{P}}
\newcommand{\E}{\mathbf{E}}

\endlocaldefs

\begin{document}
\begin{frontmatter}

\title{Asymptotics of exponential moments of a~weighted local time of a
Brownian motion with small variance}
%\author[]{\inits{}\fnm{}\snm{}\corref{cor1}}\email{}
%\cortext[cor1]{Corresponding author.}
%
%\author[]{\inits{}\fnm{}\snm{}}\email{}
%
%%\fnref{f1}
%%\fntext[]{Some remarks}
%
%\address[]{}
%\address[]{}
%
\markboth{A. Kulik, D. Sobolieva}{Asymptotics of exponential moments of a~weighted local time of a
Brownian motion}

\author[a]{\inits{A.}\fnm{Alexei}\snm{Kulik}}\email{kulik.alex.m@gmail.com}
\address[a]{Institute of Mathematics, National Academy of Sciences of \xch{Ukraine, Ukraine}{Ukraine}}

\author[b]{\inits{D.}\fnm{Daryna}\snm{Sobolieva}\corref{cor1}}\email{dsobolieva@yandex.ua}
\cortext[cor1]{Corresponding author.}
\address[b]{Taras Shevchenko National University of \xch{Kyiv, Ukraine}{Kyiv}}

\begin{abstract}
We prove a large deviation type estimate for the asymptotic behavior of
a weighted local time of $\varepsilon W$ as $\varepsilon\to0$.
\end{abstract}

%\begin{keyword} . \sep.
%\MSC[2010] . \sep.
%\end{keyword}
%

\received{9 December 2015}% Updated by VTEXPTS2LaTeX.exe, 16.03.2016
%10:16
%
\revised{22 February 2016}% Updated by VTEXPTS2LaTeX.exe, 16.03.2016
%10:16
%
\accepted{29 February 2016}% Updated by VTEXPTS2LaTeX.exe, 16.03.2016
%10:16
\publishedonline{5 April 2016}
\end{frontmatter}

\section{Introduction and the main result}

Let $\{W_t,t\geq0\}$ be a real-valued Wiener process, and $\mu$ be a
$\sigma$-finite measure on $\Re$ such that
\be\label{locfin}
\sup_{x\in\Re}\mu\bigl([x-1, x+1]\bigr)<\infty.
\ee
Recall that the \emph{local time $L_t^\mu(W)$ of the process $W$ with
the weight $\mu$} can be defined as the limit of the integral functionals
\be\label{appr}
L_t^{\mu_n}(W):=\int_0^t k_n(W_s)\, ds, \quad k_n(x):={\mu_n(dx)\over
dx}, \ n\geq1,
\ee
where $\mu_n$, $n\geq1$, is a sequence of absolutely continuous
measures such that
\[
\int_\Re f(x)\mu_n(dx)\to\int
_\Re f(x)\mu(dx)
\]
for all continuous $f$ with compact support, and \eqref{locfin} holds
for $\mu_n$, $n\geq1$, uniformly. The limit $L_t^\mu(W)$ exists in the
mean square sense due to the general results from the theory of
$W$-functionals; see \cite{dynkin}, Chapter 6. This definition also
applies to $\eps W$ instead of $W$ for any positive $\eps$. In what
follows, we will treat $\eps W$ as a Markov process whose initial value
may vary, and with
a slight abuse of notation, we denote by $\P_x$ the law of $\eps W$
with $\eps W_0=x$ and by $\E_x$ the expectation w.r.t.\ this law.

In this note, we study the asymptotic behavior as $\eps\to0$ of the
exponential moments of the family of weighted local times $L_t^{\mu
}(\eps W)$. Namely, we prove the following theorem.

\begin{thm}\label{t1} For arbitrary finite measure $\mu$ on $\Re$,
\be\label{lim}
\lim_{\eps\to0}\eps^2 \sup_{x\in\Re}\log\E_x e^{L_t^\mu(\eps W)}=\frac
{t}{2}\sup_{y\in\Re} \mu\bigl(\{y\}\bigr)^2.
\ee
For arbitrary $\sigma$-finite measure $\mu$ on $\Re$ that satisfies
\eqref{locfin},
\be\label{limsigma}
\sup_{x\in\Re}\limsup_{\eps\to0}\eps^2 \log\E_x e^{L_t^\mu(\eps
W)}=\frac{t}{2}\sup_{y\in\Re} \mu\bigl(\{y\}\bigr)^2.
\ee
\end{thm}
We note that in this statement the measure $\mu$ can be changed to a
signed measure; in this case, in the right-hand side, only the atoms of
the positive part of $\mu$ should appear. We also note that, in the
$\sigma$-finite case, the uniform statement \eqref{lim} may fail; one
example of such a type is given in Section~\ref{s3}.

Let us briefly discuss the problem that was our initial motivation for
the study of such exponential moments. Consider the one-dimensional SDE
\begin{equation}
\label{SDE} dX^\eps_t=a\bigl(X^\eps_t
\bigr)\, dt+\eps\sigma\bigl(X^\eps_t\bigr) \,
dW_t
\end{equation}
with discontinuous coefficients $a, \sigma$. In \cite{kulik-sobolieva},
a Wentzel--Freidlin-type large deviation principle (LDP) was established
in the case $a\equiv0$ under mild assumptions on the diffusion
coefficient $\sigma$. In \cite{sobolieva}, this result was extended to
the particular class of SDEs such that the function $a/\sigma^2$ has a
bounded derivative. This limitation had appeared because of formula (7)
in \cite{sobolieva} for the \emph{rate transform} of the family $X^\eps
$. This formula contains an integral functional with kernel $(a/\sigma
^2)'$ of a certain diffusion process obtained from $\eps W$ by the time
change procedure. If $a/\sigma^2$ is not smooth but is a function of a
bounded variation, this integral function still can be interpreted as a
weighted local time with weight $\mu=(a/\sigma^2)'$. Thus, Theorem \ref
{t1} can be used in order to study the LDP for the SDE \eqref{SDE} with
discontinuous coefficients. One of such particular results can be
derived immediately. Namely, if $\mu$ is a \emph{continuous} measure,
then by Theorem \ref{t1} the exponential moments of $L_t^\mu(\eps W)$
are negligible at the logarithmic scale with rate function $\eps^2$.
This, after simple rearrangements, allows us to neglect the
corresponding term in (7) of \cite{sobolieva} and to obtain the
statement of Theorem 2.1 of \cite{sobolieva} under the weaker condition
that $a/\sigma^2$ is a continuous function of bounded variation. The
problem how to describe in a more general situation the influence of
the jumps of $a/\sigma^2$ on the LDP\vadjust{\eject} for the solution to \eqref{SDE}
still remains open and is the subject of our ongoing research. We just
remark that due to Theorem \ref{t1} the respective integral term is no
longer negligible, which well corresponds to the LDP results for
piecewise smooth coefficients $a, \sigma$ obtained in \cite
{chiang-sheu1,chiang-sheu2,krykun}.

\section{Proof of Theorem \ref{t1}}

\subsection{Preliminaries}\label{s21}

For a measure $\nu$ satisfying \eqref{locfin}, denote by
\be\label{char}
f^{\nu, \eps}_t(x)=\E_x L_t^\nu(\eps W)=\int_0^t\int_{\Re}\frac
{1}{\sqrt{2\pi s\eps^2}}e^{-\frac{(y-x)^2}{2s\eps^2}}\nu(dy)\,ds, \quad
t\geq0, \ x\in\Re,
\ee
the characteristic of the local time $L^\nu(\eps W)$ considered as a
$W$-functional of~$\eps W$; see \cite{dynkin}, Chapter 6.

The following statement is a version of \emph{Khas'minskii's lemma};
see \cite{sznitman}, Section~1.2.

\begin{lem}\label{Khlem} Suppose that
\be\label{12}
\sup_{x\in\Re} f^{\nu, \eps}_s(x)\leq{1\over2}.
\ee
Then
\[
\sup_{x\in\Re} \E_x e^{L_s^\nu(\eps W)}\leq2.
\]
\end{lem}
Using the Markov property, as a simple corollary, we obtain, for
arbitrary $t>0$,
\be\label{expmom}
\sup_{x\in\Re} \E_x e^{L_{t}^\nu(\eps W)}\leq2^{1+t/s}=2e^{(\log2)(t/s)},
\ee
where $s>0$ is such that \eqref{12} holds. This inequality, combined
with \eqref{char}, leads to the following estimate.
\begin{lem}\label{l2} For a nonzero measure $\nu$ satisfying \eqref
{locfin}, denote
\[
N(\nu,\gamma)=\sup_{x\in\Re}\nu\bigl([x-\gamma, x+\gamma]\bigr),
\quad\gamma>0.
\]
For any $\lambda\geq1$ and $\gamma>0$, there exists $\eps_{\lambda,
\gamma}>0$ such that \be\label{aa}
\sup_{x\in\Re}\E_x e^{\lambda L_{t}^{\nu}(\eps W)}\leq2e^{(4 \log2)
c_0 N(\nu, \gamma)^2 t \lambda^2\eps^{-2}},\quad \eps\in(0,\eps
_{\lambda, \gamma}),
\ee
with
\[
c_0={\frac{2}{\pi}} \Biggl(1+2\sum
_{k=1}^\infty e^{-\frac
{(2k-1)^2}{2}} \Biggr)^2.
\]
\end{lem}
\begin{proof} If $\eps\sqrt{s}\leq \gamma$, then we have
\[
\ba f^{\nu, \eps}_s(x)&=\sum_{k\in\mathbb{Z}}
\int_0^s\int_{|y-x-2k\gamma
|\leq\gamma}
\frac{1}{\sqrt{2\pi v\eps^2}}e^{-\frac{(y-x)^2}{2v\eps
^2}}\nu(dy)\,dv
\\
&\leq\sqrt{c_0} N(\nu, \gamma)\sqrt{\frac{s}{\eps^2}}. \ea
\]
Take
\[
s=\bigl(2N(\nu, \gamma)\bigr)^{-2}(c_0)^{-1}
\lambda^{-2}\eps^2.\vadjust{\eject}
\]
Then the inequality $\eps\sqrt{s}\leq \gamma$ holds, provided that
\[
\eps\leq \bigl(\gamma\bigl(2N(\nu, \gamma)\bigr)^{2}c_0
\lambda^{2} \bigr)^{1/3}=:\eps_{\lambda, \gamma}.
\]
Under this condition,
\[
f^{\lambda\nu, \eps}_s(x)=\lambda f^{ \nu, \eps}_s(x)\leq
\frac{1}{2}.
\]
Now the required inequality follows immediately from \eqref{expmom}.
\end{proof}

In what follows, we will repeatedly decompose $\mu$ into sums of two
components and analyze separately the exponential moments of the local times
that correspond to these components. We will combine these estimates
and obtain an estimate for ${L_{t}^\mu(\eps W)}$ itself using the
following simple inequality. Let $\mu=\nu+\kappa$ and $p,q>1$ be such
that $1/p+1/q=1$. Then
\[
L_t^\mu(\eps W)=L_t^\nu(\eps
W)+L_t^\kappa(\eps W)=(1/p)L_t^{p\nu}(
\eps W)+(1/q)L_t^{q\kappa}(\eps W),
\]
and therefore by the H\"older inequality we get
\be\label{Hol}
\E e^{L_t^\mu(\eps W)}\leq \bigl(\E e^{L_t^{p\nu}(\eps W)} \bigr)^{1/p}
 \bigl(\E e^{L_t^{q\kappa}(\eps W)} \bigr)^{1/q}.
\ee
We will also use another version of this upper bound, which has the form
\be\label{HolA}
\E e^{L_t^\mu(\eps W)}1_A\leq \bigl(\E e^{L_t^{p\mu}(\eps W)}
\bigr)^{1/p}  \bigl(\P(A) \bigr)^{1/q}, \quad A\in\mathcal{F}.
\ee

We denote
\[
\varDelta=\sup_{x\in\Re} \mu\bigl(\{x\}\bigr).
\]
We will prove Theorem \ref{t1} in several steps, in each of them
extending the class of measures $\mu$ for which the required statement holds.

\subsection{Step I: $\mu$ is a finite mixture of $\delta$-measures}\label{s22}

If $\mu=a\delta_{z}$ is a weighted $\delta$-measure at the point $z$,
then we have
\[
L_t^\mu(\eps W)=a\eps^{-1} L_t^{(z)}(W),
\]
where
\[
L_t^{(z)}(W)=\lim_{\eta\to0}
{1\over2\eta}\int_0^t 1_{|W_s-z|\leq\eta
}
\, ds
\]
is the \emph{local time of a Wiener process} at the point $z$. The
distribution of $L_t^{(z)}(W)$ is well known; see, e.g., \cite
{ito-mckean}, Chapter 2.2 and expression (6) in Chapter 2.1. Hence, the
required statement in the particular case $\mu=a\delta_{z}$ is
straightforward, and we have the following:
\be\label{lim1}
\lim_{\eps\to0}\eps^2\sup_x \log \E_x e^{a\eps^{-1}L_t^{(z)}( W)}=\frac
{ta^2}{2}.
\ee
Note that in this formula the supremum is attained at the point $x=z$.

In this section, we will extend this result to the case where $\mu$ is
a finite mixture of $\delta$-measures, that is,
\[
\mu=\sum_{j=1}^k a_j
\delta_{z_j}.
\]
Let $j_*$ be the number of the maximal value in $\{a_j\}$, that is,
$\varDelta=a_{j_*}$. Then $L_t^\mu(\eps W)\geq\varDelta\eps
^{-1}L_t^{(z_{j_*})}(W)$, and it follows directly from \eqref{lim1} that
\be\label{lower}
\liminf_{\eps\to0}\eps^2 \sup_{x\in\Re}\log\E_x e^{L_t^\mu(\eps
W)}\geq\frac{t\varDelta^2}{2}.
\ee
In what follows, we prove the corresponding upper bound
\be\label{upper}
\limsup_{\eps\to0}\eps^2 \sup_{x\in\Re}\log\E_x e^{L_t^\mu(\eps
W)}\leq\frac{t\varDelta^2}{2},
\ee
which, combined with this lower bound, proves \eqref{lim}.

Observe that, for $\gamma>0$ small enough,
\[
N(\mu, \gamma)=\varDelta.
\]
Then by Lemma \ref{l2}, for any $\lambda\geq1$,
\be\label{kh}
\limsup_{\eps\to0}\eps^2 \sup_{x\in\Re}\log\E_x e^{\lambda L_t^\mu
(\eps W)}\leq c_1\lambda^2 t\varDelta^2
\ee
with
\[
c_1=(4 \log2)c_0={\frac{8\log2}{\pi}} \Biggl(1+2\sum
_{k=1}^\infty e^{-\frac{(2k-1)^2}{2}}
\Biggr)^2.
\]
In particular, taking $\lambda=1$, we obtain an upper bound of the form
\eqref{upper}, but with a worse constant $c_1$ instead of required
$1/2$. We will improve this bound by using the large deviations
estimates for $\eps W$, the Markov property, and the ``individual''
identities \eqref{lim1}.

Denote $\mu_j=a_j\delta_{z_j}, j=1, \dots,k$. Then
\[
L_t^\mu(\eps W)=\sum_{j=1}^k
L_t^{\mu_j}(\eps W).
\]
Fix some family of neighborhoods $O_j$ of $z_j, j=1, \dots, k$, such
that the minimal distance between them equals $\rho>0$, and denote
\[
O^j=\Re\setminus\bigcup_{i\neq j}O_i.
\]
For some $N\geq1$ whose particular value will be specified later,
consider the partition $t_n=t(n/N)$, $n=0, \dots,N$, of the segment
$[0, t]$ and denote
\[
B_{n, j}= \bigl\{f\in C(0,t): f_s\in O^j, s
\in[t_{n-1},t_n]\bigr\},\quad j\in\{1, \dots, k\}, \ n\in\{1,
\dots, N\},
\]
\[
C_{j_1, \dots, j_N}=\bigcap_{n=1}^N
B_{n, j_n}, \quad j_1, \dots, j_N\in\{1, \dots, k
\}.
\]
Observe that if the process $\eps W$ does not visit $O_j$ on the time
segment $[u,v]$, then $L^{\mu_j}(\eps W)$ on this segment stays
constant. This means that, on the set $\{\eps W\in C_{j_1, \dots, j_N}\}
$, we have
\[
L^{\mu}_t(\eps W)=\sum_{n=1}^N
\bigl(L^{\mu_{j_n}}_{t_n}(\eps W)-L^{\mu
_{j_{n}}}_{t_{n-1}}(
\eps W) \bigr).
\]
Because $L^{\mu_j}(\eps W)$ is a time-homogeneous additive functional
of the Markov process $\eps W$, we have
\[
E_x \bigl[ e^{L^{\mu_{j_n}}_{t_n}(\eps W)-L^{\mu_{j_{n}}}_{t_{n-1}}(\eps
W)}\big|\mathcal{F}_{t_{n-1}}
\bigr]=E_y e^{L^{\mu_{j_n}}_{t/N}(\eps W)} \Big|_{y=\eps W_{t_{n-1}}}.
\]
Then by \eqref{lim1}, for any $j_1, \dots, j_N\in\{1, \dots, k\}$,
\[
\limsup_{\eps\to0}\eps^2 \sup_{x\in\Re}
\log\E_x e^{L_t^\mu(\eps
W)}1_{\eps W\in C_{j_1, \dots, j_N}}\leq\frac{t}{2N} \sum
_{n=1}^N(a_{j_n})^2
\leq\frac{t\varDelta^2}{2}.
\]
Because we have a fixed number of sets $C_{j_1, \dots, j_N}$, this
immediately yields
\be\label{uploc}
\limsup_{\eps\to0}\eps^2 \sup_{x\in\Re}\log\E_x e^{L_t^\mu(\eps
W)}1_{\eps W\in C}\leq \frac{t\varDelta^2}{2}
\ee
with
\[
C=\bigcup_{j_1, \dots, j_N\in\{1, \dots, k\}} C_{j_1, \dots, j_N}.
\]
Hence, to get the required upper bound \eqref{upper}, it suffices to
prove an analogue of \eqref{uploc} with the set $C$ replaced by its
complement $D=C(0, t)\setminus C$.
Using \eqref{HolA} with $p=2$, $A=\{\eps W\in D\}$, and \eqref{kh} with
$\lambda=2$, we get
\[
\ba \limsup_{\eps\to0}\eps^2 \sup_{x\in\Re}
\log\E_x e^{L_t^\mu(\eps
W)}1_{\eps W\in D}\leq2c_1 t
\varDelta^2+{1\over2}\limsup_{\eps\to0}\eps^2 \sup
_{x\in\Re}\log\P_x (\eps W\in D). \ea
\]

By the LDP for the Wiener process (\cite{freidlin-wentzell}, Chapter 3,
\S2),
\[
\limsup_{\eps\to0}\eps^2 \sup_{x\in\Re}
\log\P_x (\eps W\in D)=-\inf_{f\in\mathrm{closure}(D)}I(f),
\]
where
\[
I(f)= \left\{ %
\begin{array}{ll}
(1/2)\int_0^t(f'_s)^2\, ds, & \hbox{$f$ is absolutely continuous on
$[0,t]$;} \\
+\infty & \hbox{otherwise.}
\end{array} %
\right.\]
For any trajectory $f\in D$, there exists $n$ such that $f$ visits at
least two sets $O_j$ on the time segment $[t_{n-1}, t_n]$. Therefore,
any trajectory $f\in\mathrm{closure}(D)$ exhibits an oscillation $\geq
\rho$ on this time segment. On the other hand, for an absolutely
continuous~$f$,
\[
|f_u-f_v|=\Biggl\llvert \int_u^v
f'_s\, ds\Biggr\rrvert \leq|u-v|^{1/2} \Biggl(
\int_0^t\bigl(f'_s
\bigr)^2\, ds \Biggr)^{1/2}.
\]
This means that, for any $f\in\mathrm{closure}(D)$,
\[
I(f)\geq\frac{\rho^2 N}{2t},
\]
which yields
\[
\limsup_{\eps\to0}\eps^2 \sup_{x\in\Re}
\log\E_x e^{L_t^\mu(\eps
W)}1_{\eps W\in D}\leq2c_1 t
\varDelta^2-\frac{\rho^2 N}{2t}.
\]
If in this construction, $N$ was chosen such that
\[
N\geq(4c_1-1)\rho^{-2}t^2
\varDelta^2,
\]
then the latter inequality guarantees the analogue of \eqref{uploc}
with $D$ instead of~$C$. This completes the proof of \eqref{upper}.

\subsection{Step II: $\mu$ is finite}\label{s23}
Exactly the same argument as that used in Section~\ref{s22} provides
the lower bound \eqref{lower}. In this section, we prove the upper
bound \eqref{upper} for a finite measure~$\mu$ and thus complete the
proof of the first assertion of the theorem. For ~finite $\mu$ and any
$\chi>0$, we can find $\gamma>0$ and decompose $\mu=\mu_0+\nu$ in such
a way that
$\mu_0$ is a finite mixture of $\delta$-measures and $N(\nu,\gamma
)<\chi$. Let $p,q>1$ be such that $1/p+1/q=1$. The measure $p\mu_0$ has
the maximal weight of an atom equal to $p\varDelta$. Since we have already
proved the required statement for finite mixtures of $\delta$-measures,
we have
\be\label{mu0}
\limsup_{\eps\to0}\eps^2\sup_{x\in\Re}\log \bigl(\E_x e^{L_t^{p\mu
_0}(\eps W)} \bigr)^{1/p}\leq\frac{t}{2}p\varDelta^2.
\ee
On the other hand, we have $N(\nu,\gamma)<\chi$ and then by Lemma \ref{l2}
\[
\limsup_{\eps\to0}\eps^2 \sup_{x\in\Re}
\log \bigl(\E_x e^{L_t^{q\nu
}(\eps W)} \bigr)^{1/q}\leq
c_1q t\chi^2.
\]
Hence, by \eqref{Hol},
\[
\limsup_{\eps\to0}\eps^2\sup_{x\in\Re}
\log\E_x e^{L_t^{\mu}(\eps
W)}\leq\frac{t}{2}p\varDelta^2+c_1q
t\chi^2.
\]
Now we can finalize the argument. Fix $
\varDelta_1>\varDelta:=\max_{x\in\Re} \mu(\{x\})^2
$
and choose $p,q>1$ such that $1/p+1/q=1$ and $p\varDelta^2<\varDelta_1^2$.
Then there exists $\chi>0$ small enough such that
\[
p\varDelta^2+2c_1q t\chi^2<
\varDelta_1^2.
\]
Taking the decomposition $\mu=\mu_0+\nu$ that corresponds to this value
of $\chi$ and applying the previous calculations, we obtain an analogue of
the upper bound \eqref{upper} with $\varDelta$ replaced by $\varDelta_1$.
Since $\varDelta_1>\varDelta$ is arbitrary, the same inequality holds for
$\varDelta$.

\subsection{Step III: $\mu$ is $\sigma$-finite}\label{s24} In this
section, we prove the second assertion of the theorem. As before, the
lower bound can be obtained directly from the case $\mu=a\delta_z$, and
hence we concentrate ourselves on the proof of the upper bound
\be\label{upper_sigma}
\limsup_{\eps\to0}\eps^2 \log\E_x e^{L_t^\mu(\eps W)}\leq\frac{t\varDelta
^2}{2},\quad x\in\Re.\vadjust{\eject}
\ee
We will use an argument similar to that from the previous section and
decompose $\mu$ into a sum $\mu=\mu_0+\nu$ with finite $\mu_0$ and $\nu
$, which is negligible in a~sense. However, such a decomposition relies
on the initial value $x$, and this is the reason why we obtain an
individual upper bound \eqref{upper_sigma} instead of the uniform one
\eqref{upper}.

Namely, for a given $x$, we define $\mu_0, \nu$ by restricting $\mu$ to
$[x-R,x+R]$ and its complement, respectively. Without loss of
generality, we assume that for each $R$, the corresponding $\nu$ is
nonzero. Since we have already proved the required statement for finite
measures, we get \eqref{mu0}.

Next, denote $M=\sup_{x\in\Re}\mu([x-1, x+1])$ and observe that $N(\nu
,1)\leq M$. Then by Lemma \ref{l2} with $\gamma=1$ and the strong
Markov property, for any stopping time $\tau$, the exponential moment
of $ L_{t}^{q\nu}(\eps W)$ conditioned by $\mathcal{F}_\tau$ is
dominated by $2e^{c_1 M^2t q^2\eps^{-2}}$.
This holds for $\eps\leq\eps_{q,1}^{x,R}$, where we put the indices
$x,R$ in order to emphasize that this constant depends on $\nu$, which,
in turn, depends on $x,R$. Since we have assumed that, for any $x,R$,
the respective $\nu$ is nonzero, the constants $\eps_{q,1}^{x,R}$ are
strictly positive.

Now we take by $\tau$ the first time moment when $|\eps W_\tau-x|=R$.
Observe that $L_{t}^{\nu}(\eps W)$ equals $0$ on the set $\{\tau>t\}$ and
it is well known that
\[
\P_x(\tau<t)\leq4\P_x(\eps W_t>R)\leq
Ce^{-tR^2\eps^{-2}/2}.
\]
Summarizing the previous statements, we get
\[
\E_x e^{L_{t}^{q\nu}(\eps W)}\leq1+2Ce^{t\eps^{-2}(c_1 M^2
q^2-R^2/2)}, \quad\eps\leq
\eps^{x,R}_{\lambda,1},
\]
which implies
\be\label{bound}
\limsup_{\eps\to0}\eps^2 \log \bigl(\E_x e^{L_t^{q\nu}(\eps W)}
\bigr)^{1/q}\leq t\bigl(c_1 M^2 q-R^2/(2q)\bigr)_+,
\ee
where we denote $a_+=\max(a,0)$. By \eqref{Hol} inequalities \eqref
{mu0} and \eqref{bound} yield
\[
\limsup_{\eps\to0}\eps^2\log\E_x
e^{L_t^{\mu_0}(\eps W)}\leq\frac
{t}{2}p\varDelta^2+t\bigl(c_1
M^2 q-R^2/(2q)\bigr)_+.
\]

Now we finalize the argument in the same way as we did in the previous
section. Fix
$
\varDelta_1>\varDelta
$
and take $p>1$ such that $p\varDelta^2\leq\varDelta_1^2$. Then take $R$
large enough so that, for the corresponding $q$,
\[
c_1 M^2 q-R^2/(2q)\leq0.
\]
Under such a choice, the calculations made before yield \eqref
{upper_sigma} with $\varDelta$ replaced by $\varDelta_1$. Since $\varDelta
_1>\varDelta$ is arbitrary, the same inequality holds for $\varDelta$.

\section{Example}\label{s3} Let
\[
\mu=\sum_{k=1}^\infty (\delta_{k^2}+
\delta_{k^2+2^{-k}} ).
\]
Then $\mu$ satisfies \eqref{locfin} and $\varDelta=1$. However, it is an
easy observation that when the initial value $x$ is taken in the form
$x_k=k^2$, the respective exponential moments satisfy
\[
\E_{x_k}e^{L_t^\mu(\eps W)}\to\E_0e^{L_t^\nu(\eps W)}, \quad k
\to \infty,\vadjust{\eject}
\]
with $\nu=2\delta_0$. Then
\[
\liminf_{\eps\to0}\eps^2 \sup_{x\in\Re}
\log\E_x e^{L_t^\mu(\eps
W)}\geq\liminf_{\eps\to0}
\eps^2 \log\E_0 e^{L_t^\nu(\eps W)}=2t>\frac{t}{2},
\]
and therefore \eqref{lim} fails.
%% Acknowledgements %%
%%%%%%%%%%%%%%%%%%%%%%
\section*{Acknowledgments} The first author gratefully acknowledges the
DFG Grant Schi~419/8-1 and the joint DFFD-RFFI project No. 09-01-14.

\end{document}